\documentclass[final]{aupl}

\newcommand{\labbel}[1]{\label{#1} [[{\bf #1}]]}  
\newcommand{\bibbitem}[1]{\bibitem{#1} [[{\bf #1}]]}  
\renewcommand{\labbel}{\label} \renewcommand{\bibbitem}{\bibitem}

\usepackage{
amsfonts,
latexsym,
amssymb,
amsmath,
amsthm,
enumerate,
color,
verbatim}

\usepackage[matrix,arrow]{xy}

 \definecolor{reeed}{RGB}{0,0,25}

\newcommand{\arxiv}[1]{{\color{reeed}#1}}

\numberwithin{equation}{section}

\newtheorem{theorem}{Theorem}[section]

\newtheorem{proposition}[theorem]{Proposition}

\newtheorem*{proposition*}{Proposition} 
\newtheorem*{claim*}{Claim}

\newtheorem*{theorem*}{Theorem}
\newtheorem*{corollary*}{Corollary}

\theoremstyle{definition}
\newtheorem{definition}[theorem]{Definition}
\newtheorem*{definition*}{Definition}

\theoremstyle{remark}
\newtheorem{remark}[theorem]{Remark}
\newtheorem*{remark*}{Remark}
 
\newtheorem{example}[theorem]{Example}

\DeclareMathOperator{\Lim}{Lim}
\DeclareMathOperator{\inflim}{inf\hspace{0 pt}lim}

\newcommand{\narr}{{($\overline{ \text{N}}$)}}

\begin{document}

\title{Noncommutative infinitary semigroups}

\author{Paolo Lipparini} 
\address{
Dipartimento di Matematica. Viale della Semi-Ricerca Scientifica,
II Universit\`a di Roma (Tor Vergata), I-00133 ROME, ITALY (currently retired) \\ORCiD 0000-0003-3747-6611}
\email{lipparin@axp.mat.uniroma2.it}
\urladdr{\\http://www.mat.uniroma2.it/\textasciitilde lipparin}

\thanks{Work performed under the auspices of G.N.S.A.G.A. Work 
 supported by PRIN 2012 ``Logica, Modelli e Insiemi''.
The author acknowledges the MIUR Department Project awarded to the
Department of Mathematics, University of Rome Tor Vergata, CUP
E83C18000100006.}

\begin{abstract}
Various kinds of infinitary operations satisfying forms
of associativity have been considered in the literature 
by various authors, including
A. Tarski, C. Karp, J. H. Conway, D. Krob, 
N. Bedon,  and C. Rispal.
 Applications include the arithmetics 
of binary relations, 
quasigroups,
and automata theory,

We present a general definition for an infinitary
noncommutative partial 
semigroup; the definition  extends and encompasses all the previous
notions. In particular, we show
that new phenomena occur in the noncommutative case,
giving rise to a somewhat richer (and, by the way, more difficult)
theory.
\end{abstract}



\maketitle

\section{Introduction} \labbel{intr}

The notion of a complete (commutative) monoid,
sometimes  introduced with different terminology,
dates back at least to Krob \cite{K} and implicitly appears in
Conway \cite{C}.  Commutative monoids
with partially defined infinitary products
appeared in Higgs \cite{H}  and in Manes and  Benson \cite{MB}.
 See  Hebisch  and    Weinert \cite{HW} for a survey.
An abstract treatment of 
operations depending on countably many arguments and
satisfying general forms of associativity
appeared earlier  in  Tarski \cite{Tcard} and, even
in the noncommutative case, in \cite{Tord} for sequences 
indexed by the set $ \omega$ of the natural numbers, as a basic
framework for a general study of the arithmetics 
of binary relations. 
General forms of associativity for sequences 
are implicit in the classical theory of 
series convergence.
Some comparatively weak associative properties
with quite strong consequences and essentially dealing 
with products of sequences indexed by $ \omega+ \omega $ 
have been considered in Madevski, Trpenovski and {\v C}upona \cite{MTC}
and 
Belousov and  Stojakovi\'c \cite{BS} 
in connection with infinitary quasigroups.

Another approach 
to general  associativity has been inspired by the study
of  strings, aka words, 
indexed by ordinals.
Some kind of  semigroup structure on such sets of words has long been 
recognized \cite[Chapter II]{Karp}.
Some special cases, e.~g., 
$ \omega$-indexed words
or  words indexed by the order-type $\zeta$ 
of the integers have been thoroughly analyzed
in connection
with techniques and ideas from automata theory.
See Perrin and Pin \cite{PP} for further details.
See also, e.~g.,  Choffrut and  Horv{\'a}th \cite{CH} for references about
transfinite strings.

Bedon \cite{Be}
and Bedon and Carton \cite{BC} gave the definition 
of a (noncommutative) $ \omega_1$-semigroup,
in which all products  of sequences
indexed by  countable ordinals
are defined (in this note we will use the expression $ {<}\omega_1$-semigroup
for this notion, instead).
Words indexed by countable linear orderings
have been considered by Bruy\`ere and Carton; see \cite{BC2}.
 A corresponding abstract notion, that of a $\diamond$-semigroup,
in which all  products of sequences indexed by 
 countable  scattered  linear orders are defined,
appears in  Rispal 
\cite{Ri} and 
 Rispal and Carton \cite{RC}.
See also Bedon and Rispal \cite{BR}. 

We refer to all the quoted sources
for further details and references,
in particular about earlier works. 
Though we have performed a quite extensive search,
the literature probably contains many further  important works
on the subject. 

The aim of the present note is to fully generalize all the above notions
 in the noncommutative case and in the case of products depending on
arbitrarily many arguments.
In Definition \ref{pnc} 
we introduce the notion of a partial infinitary semigroup;
it is a structure with a partial ``product operation''  defined on sequences
indexed by arbitrary linearly ordered sets; this  operation
is required  to satisfy a general 
form of associativity.
The definition encompasses all the notions mentioned above.
 In general,
 we will work in the context of semigroups,
rather than monoids, namely, the existence of ``neutral elements'' 
is not necessarily assumed.

In the rest of Section \ref{defex} 
we  list some easy consequences
of Definition \ref{pnc}, together
with some related remarks and definitions.
Section \ref{secex} is devoted to examples 
which are not commutative. More constructions
are presented in Section \ref{const}.
The connection with the commutative case is analyzed in 
Section \ref{commu}. There we show that
the noncommutative case allows  much richer situations.
For example, Krob \cite{K} showed that every complete
commutative semigroup has  an absorbing element,
while this is not always true in the noncommutative case.
Moreover, it is well-known that in 
the complete commutative case we cannot have some ``infinitary identity''  $e$
together with some other elements such that $a*b=e$,
while this is possible in the noncommutative case (Theorem \ref{abe}).
Finally, in Section \ref{prob} we discuss
some possible variations of the main notion
and ask a few problems.

\section{Preliminaries} \labbel{prel}

Our notation is standard. 
In general, by an \emph{operation}, possibly infinitary,
we mean a \emph{partial operation}, that is, the result of the operation
might be   undefined in some cases.
However, if not otherwise specified, a (finitary) semigroup
is always assumed to be endowed with a total operation. 
In detail, a \emph{semigroup  ``in the classical sense''}
is a set
together with a \emph{total} 
binary associative operation.
Again if not otherwise mentioned, 
 a \emph{function} $f: I \to J$  
is always assumed to  be \emph{total},
 that is, $f(i)$ is always defined,
for every $i$ in the domain $I$ of $f$.  
Composition of functions
is written as if we were drawing diagrams
like
$I \stackrel{f}{\rightarrow}J \stackrel{g}{\rightarrow}K$. 
In formulas,
$(f \circ g)(i) $ means $  g(f(i))$.  

In some examples we will assume familiarity
with ordinals, see e.~g.
 Bachmann \cite{bach}.
\arxiv{Also, Jech \cite{J}, 
 Mendelson \cite{Men} or Sierpi{\'n}ski \cite{sier}.
Also textbooks more specifically 
devoted to ordered sets usually treat ordinals,
 e.~g., Harzheim \cite{Harz} and Rosenstein \cite{Ro}. 
Throughout, $\alpha$, $\beta$ \dots denote ordinals
and $\lambda$, $\mu $\dots denote infinite cardinals.
As usual, a cardinal is identified with the smallest ordinal
having its cardinality.
This is possible since we are always assuming the Axiom of Choice.
}
Note that in the theory of ordinals 
the set $\mathbb N$ of natural numbers is 
denoted by (or, better, identified with) $ \omega$, 
the first infinite ordinal, equivalently, cardinal.
We will generally try to follow the most usual convention
in each particular setting. The reader should be warned that, as
far as the present note is concerned, $\mathbb N$ and $ \omega$
denote exactly the same object.
Also, at  first reading, the reader might
always think of the  notions and the results presented here
as restricted to countable ordinals and sets.
On the opposite direction,
we will state some definitions in terms of 
such large objects as (possibly proper) \emph{classes}. See
e.~g. \cite{J}.
The reader not worried by foundational issues might always
suppose that ``class'' and ``set'' are synonymous words, say, to mean
``collection''. 
\arxiv{Details for the treatment of classes are presented in some 
of the above mentioned books. See \cite{FBL} 
for a full discussion and  Kanamori \cite{Ka}, Tarski \cite[p. 154]{Tcard} 
for further remarks. 
We have listed these references for the
reader who wants to be fully comfortable with respect to 
foundational  issues; of course, they are 
not strictly necessary in order to  understand the present paper.}

\section{The basic notions} \labbel{defex} 

As customary when dealing with non commutative operations,
we will generally use the multiplicative notation instead of the additive 
notation.
In the noncommutative case the order of the factors in a product is relevant,
hence, to give 
the definition of some infinitary product, we have to deal with 
sequences indexed by a linearly ordered 
 set. 
\arxiv{In many cases we will be
concerned only with well-ordered index sets, however the  definition 
seems to have some interest in the general case of  linearly ordered 
index sets. See, e.~g.,  Examples \ref{diamond},
 \ref{ordex}, \ref{labeledord}, \ref{rel}, \ref{addinf} -  \ref{pointw} and
Theorems  \ref{krob}, \ref{abe}.}

Throughout the present note, if not otherwise specified, 
$I$ and $J$ are intended to be 
possibly empty linearly ordered
sets---reference to the order will be made explicit only if  necessary.
If not otherwise and specifically mentioned, \emph{order} and
\emph{ordered set} are synonymous with \emph{linearly ordered set}.
In some examples we will deal with 
partially ordered sets; in such a situation we will
use the expression \emph{poset}.
For the sake of brevity, we  will call
a sequence indexed by a linearly ordered set
 simply a \emph{linearly-indexed sequence}.
 When we speak 
of a \emph{well-ordered sequence}, we intend that the corresponding 
linear order is a well-order.
An \emph{order preserving map}   (or \emph{ordermorphism}, or
\emph{homomorphism of ordered sets}) is
 a  function which preserves the order relations, namely,
$a \leq b$ implies $f(a) \leq f(b)$.
An \emph{order isomorphism} is a bijective order preserving function
such that also its inverse is order preserving.

\begin{definition} \labbel{pnc}
 A  \emph{partial infinitary semigroup}, or 
sometimes simply an \emph{infinitary semigroup} for the sake of brevity,
 is a nonempty class $S$ 
together with a class function $ \prod$
whose codomain is $S$ itself and
whose domain
 consists of a class of linearly-indexed sequences
 of elements of $S$.
The image of $(a_i) _{i \in I} $  under $\prod$, when defined,
will be denoted by $ \prod _{i \in I} a_i $.
If $a_i=a$, for every $i,j \in I$,
we will simply write   $ \prod _{I} a $,
and $\prod _ \emptyset $ in the case of the empty sequence.

A partial infinitary 
semigroup is required to satisfy the following two properties.
  \begin{enumerate}
\item[(U)]
\emph{If $I= \{ i  \} $ has one element,
then 
$\prod _{i \in I} a_i $ is defined 
and is equal to $a_i$.}
   \item[(N)]
\emph{Whenever $\prod _{i \in I} a_i $ is defined
and $\pi:I \to J$ is a surjective order preserving map,
then all products in the following equation are defined
\begin{equation*}\labbel{N}
\prod _{i \in I} a_i = \prod _{j\in J} \prod _{\pi(i) =j }  a_i
  \end{equation*}     
and equality actually holds.}
   \end{enumerate} 

In the above-displayed formula 
the meaning of $ \prod _{\pi(i) =j }  a_i$
is $ \prod _{i \in I_j }  a_i$, where
$I_j  = \{ i \in I \mid \pi(i) =j \}  $ is given the (sub-)order 
induced by the order on $I$. Note that 
$I_j$ is a \emph{convex} subset of $I$, 
that is, if $a <b \in I_j$ and $a< c<b$ holds in $I$, then
$c \in I_j$.  This follows from the assumption 
that $\pi$ is  order preserving. 

As usual, when there is no risk of ambiguity, we will simply write
$S$ in place of $(S, \prod)$. We will usually and 
informally refer to $\prod$ as an \emph{operation},
though formally, in the more usual sense, it is a class of (partial) operations,
one for each ordered set (recall that in this paper
``ordered set'' always means ``linearly ordered set'').
\end{definition}

The cases we will be most interested in are considered in the next definition.

\begin{definition} \labbel{pncx}    
 In case 
$\prod _{i \in I} a_i $ is defined for every nonempty  linearly-indexed 
(respectively, nonempty well-ordered) sequence of elements of $S$,
we say that $S$ is a 
\emph{complete  semigroup}
(respectively, an
\emph{ordinal semigroup}).
Notice that in this note we always mean \emph{complete}
in the infinitary sense, namely, a finitary semigroup with a total
binary  operation
is not complete in our sense. 
 \end{definition}   

\begin{remark} \labbel{conseqa}
 It follows immediately from properties (N) and (U) that 
every partial infinitary semigroup satisfies the following property
  \begin{enumerate} 
\item[(Iso)]
\emph{If 
$I$ and $J$ are isomorphic ordered set,
$f:I \to J$ is an order isomorphism
and $b _{f(i)}  = a_i$, for every $i \in I$, 
then 
$\prod _{i \in I} a_i $ is defined 
if and only if 
$\prod _{j \in J} b_j $ is defined and, in case they are defined,
they are equal.} 
   \end{enumerate}

 \arxiv{
Note that 
$\prod _{j \in J} b_j $
is the same as
$\prod _{j \in J} a _{g(j)}  $, where
$g$ is the inverse of $f$. 

It follows from Property (Iso) that, in order to construct some 
partial infinitary semigroup $S$, it is enough
to define $\prod _{i \in I} a_i $ (or say that the product 
is not defined) 
on just one representative $I$ of each isomorphism class of linear orders,
for every $I$-indexed sequence.
It is intuitively clear that if 
some form of Property (N) is satisfied relatively to some
family of such representatives,
then the operation can be expanded uniquely
using (Iso)
in order to get a partial infinitary semigroup 
in the sense of Definition \ref{pnc}.
 Put in another way, 
infinitary semigroups can
be thought to be defined on (sequences indexed through) order-types,
rather than orders. 
Of course, Property (N)
needs to be checked for coherence in the general case.
Full details in the case of ordinal semigroups
will be given in Appendix I. 
}
 \end{remark}

\arxiv{
We now describe some further special cases of Definition \ref{pnc}. 

\begin{definition*} 
(I) As a cardinal-restricted notion, if $\kappa$ is an infinite cardinal and 
$\prod _{i \in I} a_i $ is defined for all nonempty linearly-indexed 
 sequences having index set of cardinality $<\kappa$,
we say  
that $S$ is a
\emph{${<}\kappa$-complete  semigroup}.
Note that, in parallel situations, many authors say \emph{$\kappa$-complete}
for what we call  ${<}\kappa$-complete.
For example, a filter is usually said to be 
{$\kappa$-complete} if it is closed 
under the intersection of $<\kappa$ members. 
See \cite{J,Ka}.
Any $\kappa$-complete filter in the above sense  
is a ${<}\kappa$-complete semigroup in the present terminology, 
 with the operation of 
intersection (of $< \kappa$ many sets).

(II) If $\gamma$ is an infinite ordinal,
a  \emph{${<}\gamma $-semigroup}
is a partial infinitary semigroup such that 
$\prod _{i \in I} a_i $ is defined for all nonempty well-ordered
 sequences  of order-type $<\gamma$. 
If $\prod _{i \in I} a_i $ is defined whenever 
$I$ is nonempty of order-type  $\leq \gamma$, we speak of a
 \emph{${\leq}\gamma $-semigroup}.
Thus this is the same as a \emph{${<}(\gamma +1)$-semigroup}.
Note that a ${<}\omega$-complete semigroup
is the same as a ${<}\omega$-semigroup, since every finite
order is well-ordered.
\emph{Countably complete}
could be a possible term for a
${<} \omega _1$-complete semigroup.
Recall that $ \omega_1$ is the collection
of all countable ordinals and can be taken as the standard
representative
for the smallest uncountable cardinal. 

(III) In the above definition of, say, an ordinal semigroup, we leave open
the possibility that $\prod _{i \in I} a_i $ is  defined 
even in some cases in which $I$ is not well-ordered. Thus,
for example, according to the above definitions, 
any complete semigroup is also an ordinal semigroup.
 We will sometimes have occasion
to consider a structure in which
 $\prod _{i \in I} a_i $ is defined \emph{exactly} 
in case $I$ is a well-order. In this case we will speak of
an \emph{ordinal semigroup in the strict sense}. 
A similar convention applies to ${<}\kappa$-complete  semigroups
and  to ${<}\gamma $-semigroups.

Note that all  the above definitions  make sense, 
namely, (N) is always applicable, since,
for example in the case of $< \gamma $ semigroups,
 if $ \gamma ' < \gamma  $  and 
$\pi: \gamma ' \to J $ is surjective,
then $J$ and $ \pi^{-1}( \{  j \}) $, for $j \in J$,
are all nonempty and well-ordered of type $< \gamma $.
Compare the statement of Property (Ord)
in Remark \ref{conseqord}.

(IV) We could define a stronger form of, say,
 ${<}\kappa$-completeness in the sense that
whenever $( a_i) _{i \in I} $ is
a sequence, the linear order $I$ is partitioned into
 ${<}\kappa$ convex nonempty sets $I_j$ ($j \in J$)
and  $ \prod _{i \in I_j }  a_i$ is defined, for 
 every $j \in J$,
also $ \prod _{i \in I}  a_i$ is defined.
A similar stronger definition can be given relative to 
${<}\gamma $-semigroups.
We will not need these stronger definitions here;
in particular, our main emphasis is on 
complete  and
ordinal semigroups, so that the cardinal or ordinal-related 
definitions will be essentially used only
in exemplifications. In any case, essentially all the examples we will present
are ${<} \omega $-complete in the above stronger sense. Compare
Remark \ref{trivi} for pathological examples of semigroups
satisfying Definition \ref{pnc} .
Compare also Conditions \narr\ 
in Remark \ref{narr} and (Eq)
in Remark \ref{pvar}.
\end{definition*}   
}

\begin{remark} \labbel{semig} 
Of course, if $n \geq 0$ and $I= \{ 0,1, 2, \dots, n\} $ with 
the usual ordering given by
$0<1 < 2 < \dots <  n$,
we write $a_0*a_1*a_2 * \dots *a_n$  or $a_0a_1a_2  \dots a_n$ in place of
$\prod _{i \in I} a_i $. 
The notation is not ambiguous, by
(Iso) from Remark \ref{conseqa} and,
if $n=0$,
 by property (U). By property (N),
if $a_{0}*a_{1}*a_{2} * \dots * a_{n}$ is defined, then all 
(contiguous) subproducts are defined.
Moreover, by iterating (N),
we can associate terms at will 
inside the above product and
we  get the same outcome in each case.
In particular, a (total binary) semigroup in the classical sense
is essentially the same as a 
${<} \omega $-semigroup, namely an infinitary semigroup
in the sense of Definition \ref{pnc} such that 
$\prod _{i \in I} a_i $ is defined exactly when $I$ 
is finite.

Indeed, if $S$ is an
${<} \omega $-semigroup in the sense of
Definition \ref{pncx}(II),
then 
$a_0 * a_1 = \prod _{i=0,1} a_i $  
defines a total binary operation on $S$ and this operation
 is shown to be associative by applying (U) and 
(N) in two ways to $\prod _{i =0,1,2} a_i $. 
The previous  paragraph then shows that the notation
$a_{0}*a_{1}*a_{2} * \dots * a_{n}$  
 has the same meaning either when it is 
 considered in the sense of classical semigroup theory,
or as an abbreviation
for $\prod _{i \in I} a_i $, with
$I= \{0, 1, 2, \dots, n\} $.

On the other hand, given a semigroup in the classical sense,
the expression $a_{0}*a_{1}*a_{2} * \dots * a_{n}$  is uniquely
defined, for every $ n \geq 0$,
by (finite) generalized associativity.
By taking 
$a_{i_0}*a_{i_1} * \dots * a_{i_n}$
as a definition for 
 $\prod _{i \in I} a_i $, for
$I= \{ i_0, i_1, \dots,i_ n\} $
with the ordering
$ i_0 < i_1 < \dots  < i_ n$, we have that
Property (U) is true by definition 
and, again, generalized associativity shows that Property 
(N) holds. 

Henceforth we will
not distinguish between the notions of a semigroup 
(in the classical sense)
and of an 
${<} \omega $-semigroup 
\arxiv{ (in the strict sense,
cf. Definition \ref{pncx}(III)), }
 though formally they are distinct
notions.
In the former case we have
just one binary operation,
while in the latter case we have 
countably many operations:
one $n$-ary operation
for each $n > 0$. 
\end{remark} 

\begin{remark} \labbel{trivi}
(a) To be pedantic, some special trivial cases of structures 
satisfying Definition 
\ref{pnc} 
are not really ``semigroups''.
If $\prod _{i \in I} a_i $ is defined exactly 
when $ \vert   I  \vert   =1$, we get the trivial unary identity operation,
by (U).
If $\prod _{i \in I} a_i $ is defined exactly 
when $ \vert   I  \vert   $ is either $1$ or  $2$, we essentially get
a binary operation, required to satisfy no additional property.

(b) On the other hand, if $\prod _{i \in I} a_i $ is defined for every  
$I$ such that  $ \vert   I  \vert = 3  $,
then, by the  arguments in  Remark \ref{semig},  
we get (at least) the structure of a semigroup in the classical sense.
In our terminology, we can expand 
$\prod$ to get even a ${<} \omega $-semigroup.

(c) In a sequel \cite{ordsmg} to the present paper we will
present an infinitary generalization of (b) above.
In detail, we will prove that if 
$\gamma$ is an infinite ordinal and some
infinitary product is defined 
exactly for \emph{all}  sequences indexed by ordinals $\leq \gamma$, then 
such a product
can be uniquely expanded to apply to  
all sequences indexed by ordinals 
having the same cardinality of $\gamma$.
More formally,
any ${\leq} \gamma$-semigroup 
can be uniquely expanded (without adding further elements)
to a ${<} \mu  $-semigroup,
 where $ \mu = \vert   \gamma  \vert  ^+$ is
the cardinal immediately larger than the cardinality of $\gamma$.
 \end{remark} 

 \arxiv{
\begin{remark*}    
We have given Definition \ref{pnc} allowing 
 proper classes just for simplicity.
The reader might always think of every notion
 we deal with as restricted to ordered index sets 
having cardinality less than
some fixed cardinal, as in Definition \ref{pncx}(I)(II),
thus generally avoiding any possible set-theoretical problem\footnote{Strictly
 speaking, in order to avoid proper classes,
 we should require that the ordered sets under consideration
are \emph{subsets}  of some fixed cardinal.}. This approach
 will be explicitly mentioned  in
each particular case. 
See Examples \ref{ord}(b) - 
\ref{ordex}(b).    
Actually, as we mentioned, at first reading
the reader might even assume 
that  all of our definitions and notions
are restricted to  countable indexed  sets, still
getting some interesting results (at least in the author's opinion).
Stating the results for classes
presents no serious problem
and surely simplifies definitions 
and statements\footnote{To be pedantic, Definition \ref{pnc}, as stated,
might appear to be not correct even working in a set theory 
allowing proper classes. The definition can be reformulated in a 
safe way by letting an infinitary semigroup  be a class of triplets
of the form $(S,I,\prod_I)$, where $S$ is the same set for each triple,
each $I$ is a linearly ordered set and each $\prod_I$ is a partial function 
from $I$ to $S$.}. 
 \end{remark*} 
}

For  sequences indexed by ordinals $ \leq \omega$, the conjunction of 
Postulates II and II$'$ considered in
Tarski \cite{Tord}, though formally
slightly different, is
 equivalent to (N), via Postulate I. 
\arxiv{ Compare Remark  \ref{conseqord} below. }
However, the main emphasis
in \cite{Tord} is on objects  with more structure
and satisfying
further conditions. 
As we mentioned, a definition bearing some resemblance to 
Definition \ref{pncx} appears in \cite{K},
but the notions from \cite{K} imply  very strong forms
of commutativity; see Section \ref{commu}.
On the other hand, the authors of \cite{Ri,RC} 
give a  definition 
essentially equivalent to 
Definition \ref{pncx},
though they restrict themselves
only to countable  scattered linear orders.
Their definition  can be  extended without any special effort to apply 
to arbitrary linearly ordered sets 
and in this way  turns out to be equivalent to 
the notion of a complete semigroup as 
introduced in Definition  \ref{pncx}.

\begin{remark} \labbel{conseqc}
 Note that  it  follows from property (N) that
if $\prod _{i \in I} a_i $ is defined and $H$ is a convex subset of $I$,
then  $\prod _{i \in H} a_i $ is defined, too.

\arxiv{
Let us say that a partial infinitary semigroup $S$ is 
\emph{$I$-complete} if
$\prod _{i \in I} a_i $ is defined,
for every $I$-indexed sequence of elements of $S$.
 We then get that if $S$ is $I$-complete, then 
$S$ is $J$-complete, whenever $J$ is a convex subset of $I$. 

 As a final remark, we note that (N) is equivalent to the
conjunction of (Iso) and of the weaker version of (N)
in which $\pi$ is the canonical projection into 
a quotient $J$ given by an equivalence relation on $I$ 
with convex equivalence classes. }
\end{remark}

\arxiv{ 
\begin{remark} \labbel{conseqord}
 In comparison with arbitrary (linear) orderings, 
the case of ordinal semigroups  is simpler, due to the
facts that each well-ordered set is order-isomorphic with
a unique ordinal and that  ordinals have no
nontrivial automorphisms. 
In fact, with reference to   Definition \ref{pncx}, 
a more appropriate name for what we call an ordinal-semigroup would
have been a \emph{well-order-complete semigroup}. 
The fact that it is enough to define 
the operations on the ordinals 
justifies the shorter  terminology.

Now for some more details.
Recall that  every well-ordered set
is order-isomorphic to an ordinal or, in other words,
the ordinals form a class of representatives
for the class of well-ordered types.
Hence, in order to construct an 
ordinal semigroup, in view of (Iso) it is enough to define
$\prod _{i < \alpha } a_i $ for $\alpha$ an ordinal.
Property (N) then translates to the following condition
\begin{enumerate}
   \item[(Ord)]
\emph{
Assume that  $\prod _{ \alpha  <\delta } a_\alpha  $ is defined,
 $\pi: \delta  \to \eta$ is a surjective  order preserving map  and,
 for every $j \in \eta$,
let $I_j = \{ \alpha < \delta \mid \pi ( \alpha ) = j \}
= [\beta_j, \beta_j + \varepsilon_j )  $.
Whenever the above assumptions are met,
we require that  
all the products  in
the following equation are defined
\begin{equation*}
\prod _{ \alpha  < \delta } a_ \alpha  = \prod _{j < \eta} 
\prod _{ \gamma < \varepsilon_j } a _{\beta_j + \gamma }
  \end{equation*}     
and equality actually holds.
}
   \end{enumerate} 

It is intuitively clear that if 
 $S$ is a class, $\prod$ is a partial
class operation defined on ordinal-indexed sequences
of elements of $S$ and
 Properties  (U) 
from \ref{pnc} and (Ord) above are satisfied,
then, using (Iso), $S$ can be uniquely expanded to 
some partial infinitary semigroup
in the sense of Definition \ref{pnc}. 
Details will be presented in Appendix I.
In any case, here ordinal semigroups will be mainly considered
only as examples. 

\end{remark} 

Though not exactly our main object of study, 
we will sometimes have occasion to
 talk of homomorphisms of infinitary semigroups.
For partial algebras, 
be them finitary or not,
 there is not a unique notion of homomorphism, see
\cite[Chapter 2]{G}.
We will use the weakest one.

\begin{definition} \labbel{hom}
If $(S, \prod)$ and $(T, \prod')$ are partial infinitary semigroups,
we say that a (total) function $f:S \to T$ is a \emph{homomorphism} if  
$\prod' _{i \in I} f(a_i) $ is defined whenever
$\prod _{i \in I} a_i $ is defined and, if this is the case,
 then
$f \left( \prod _{i \in I} a_i\right)= \prod' _{i \in I} f(a_i) $. 

If $S \subseteq T$ and the inclusion is a homomorphism,
we say that $(T, \prod')$ is an \emph{extension} of $(S, \prod)$.
 If, in addition,
$S=T$,  we talk of an \emph{expansion}.
Thus ``expansion''  means that only the operation    
is expanded, but the base set does not change.
 \end{definition}   
}

\section{Examples} \labbel{secex} 

As we mentioned at the beginning,
there are many examples of commutative infinitary
semigroups. See also Section \ref{commu}.
Since our main interest here is in the noncommutative case, we will
first provide examples of infinitary semigroups which are not 
(more exactly, cannot be made) commutative. 

We are mainly concerned with examples which are
complete semigroups or at least ordinal semigroups.
In the following examples, 
when dealing with 
classical notions in which 
the operations are usually written in additive notation,
we will retain 
the standard use,
instead of  shifting to the 
multiplicative notation.

\begin{example} \labbel{reals}  
($\bullet$)    The real numbers with the usual finitary addition and the 
infinitary operation
$\sum _{i \in \mathbb N } a_{i}$
defined exactly when the series is convergent in the classical sense
is a partial infinitary semigroup.
Indeed, if some series is convergent, we can group together
as many \emph{adjacent} summands as we want,
still getting the same limit. 
If one wants to formally match Definition \ref{pnc},
the operation should be 
 expanded using (Iso) from Remark \ref{conseqa}.
\arxiv{Compare Remark   \ref{conseqord}.}

Note that, unless the series
is absolutely convergent, we are not allowed to change the ordering
of the factors, 
hence this infinitary semigroup is not commutative
in any reasonable sense.
The classical definition of
an infinitary commutative semigroup will
be recalled in Section \ref{commu}.

($\bullet$) Of course, the  above example can be extended to series of 
complex numbers, series of $n$-tuples  of real numbers 
(with fixed $n$ and pointwise convergence), series of functions, 
and even to surreal numbers (see \cite[VIII, Definition 3.20]{Sieg}).
A very general setting is a  Hausdorff topological group,
where  $\sum _{i \in \mathbb N } a_{i}$ is 
defined exactly if the sequence of the partial sums is converging,
in which case the value of the sum is taken to be the limit.
If $\sum _{i \in \mathbb N } a_{i}$ converges to $a$,
then, by continuity, $\sum _{i \geq 1 } a_{i}$ converges to $a - a_0$.
Thus, for example, $\sum _{i \in \mathbb N } a_{i} = a_0 + \sum _{i \geq 1 } a_{i}$.
All the other cases of (N) are similar.

While there might be possibly interesting
examples in which the binary operation is just
a semigroup operation, we do need a group operation in 
order to carry over the above argument.
For example, real numbers with the product do not form a partial infinitary semigroup,
taking the limit of partial products as the infinitary operation.
The limit of the sequence $0$, $0 \cdot (-1)$, $0 \cdot (-1) \cdot (-1)$, 
$0 \cdot (-1) \cdot (-1)\cdot (-1) $,  \dots\ 
is $0$, but the sequence  
$-1$, $(-1) \cdot (-1)$, 
$\cdot (-1) \cdot (-1) \cdot (-1) $, \dots\ has no limit, so that Property (N)
fails. Of course, on the other hand, the set of \emph{nonzero} real or complex numbers
 do form a partial infinitary semigroup
(or, simply, consider as undefined any infinite product with some
null factor).

\arxiv{
($\bullet$) The above examples cannot be extended to 
a complete (infinitary) semigroup, actually,
not even to a $ {\leq}\omega$-semigroup.
Were this the case, the series
$-1+1+(-1)+1+(-1) \dots$ would have some definite value $ \lambda $,
but then (N) implies that $ \lambda $ is equal both to   
$(-1+1)+(-1+1)+ \dots = 0+0+ \dots =0$
and also to
$-1+(1-1)+ \dots =-1+ 0+0+ \dots =-1$,
a contradiction. Of course, this is a standard classical argument;
we report it here in order to illustrate the meaning of (N).

($\bullet$) As usual, instead of $\mathbb R$, we could have considered
$\mathbb R \cup \{ \infty, -\infty \} $, with the standard
conventions about convergence. 
However, if we allow some $a_i$ to be $\infty$,
some caution is needed. For example, it would be natural to
set  $ \infty + (-1)+1+(-1)+1+(-1) \dots = \infty$, but with this position
$\mathbb R \cup \{ \infty, -\infty \} $ is \emph{not} an infinitary semigroup 
in the sense of Definition \ref{pnc}. In fact, Property (N)
and Remark \ref{conseqc} would then imply
 that $-1+1+(-1)+1+(-1) \dots$ is defined, but we have
seen that this leads to a contradiction.

A possibility (among many other ones)
for making $\mathbb R \cup \{ \infty, -\infty \} $
a partial infinitary semigroup in the sense of 
Definition \ref{pnc} is the following.
Sums whose summands are real numbers
are treated classically.
Any sum (finite or infinite) containing both $\infty$ and $-\infty$
as summands is not defined.

 Next, suppose that 
a sum contains at least one $\infty$ and no $-\infty$ 
(the symmetric case is treated the same way).
Any such finite sum is set to $\infty$.
 Any infinite sum in which $\infty$ occurs infinitely many times
is set to $\infty$, as well.
An infinite sum 
with only a finite number of occurrences of $\infty$ 
has value $\infty$ if the subseries 
$\sum _{n<i < \mathbb N} \alpha_{i}$ 
converges either to some real number, or to $\infty$,
where $n$ is the largest place at which $\infty$ 
occurs. Note that $\sum _{n<i < \mathbb N} \alpha_{i}$ 
is a series of real numbers.
Otherwise the sum is undefined.
In the last condition,
the case in which
$\sum _{n<i < \mathbb N} \alpha_{i}$ 
converges to $-\infty$
should be excluded, since, otherwise,  
$\infty - \infty$ should be defined, by Remark \ref{conseqc}.
We would have, say, $\infty -1-1-1\dots = \infty + (-1-1-1\dots) = \infty - \infty$. 

($\bullet$) If we consider
$\mathbb R^{\geq 0} \cup \{ \infty \} $, instead, we get
a complete semigroup, which turns out to be commutative,
hence not an example we are particularly interested here. 
}

($\bullet$) Of course, we can iterate the above examples through the transfinite, 
so that, say, some sums of ordinal-indexed sequences of real numbers 
turn out to be defined.
Note that, generally, only the case of absolutely convergent sequences
is considered in the literature (in which case the ordinal-indexed notion 
essentially reduces to the $ \mathbb N$-indexed notion), while, as we mentioned, 
the definitions   here do not need to assume absolute convergence. 
\end{example} 

Also many notions of limit can be considered
as infinitary semigroups.

\begin{example} \labbel{limits}
(a) The limit operation (for finite and  $ \mathbb N$-indexed sequences)
 in a Hausdorff
topological space can be considered as an infinitary
semigroup, provided we consider the last element of a finite sequence
as the limit of the sequence. Then Property (N)
is a restatement of the fact that if some $ \mathbb N$-sequence converges,
then any infinite subsequence converges to the same limit.

(b) The situation for general linearly ordered sequences is more
problematic.
If $( a_i) _{i \in I} $ is a sequence in a 
Hausdorff topological space, with $I$ linearly ordered,
the standard definition of limit is given by  
$\lim  _{i \in I} a_i = a$ if, for every neighborhood $U$ of
$a$, there is $j \in I$ such that   $a_i \in U$,
for every $i \geq j$ (of course, this is a very special case
of net convergence).

However, the above operation of limit does not give
a partial infinitary semigroup since, say,
if $I= \omega + \omega $, then an $I$-indexed sequence
might converge, while the first $ \omega$th part
of the sequence might diverge, thus (N) fails
(actually, an $\omega + \omega $ indexed sequence
converges if and only if the second $ \omega$th part converges).  
In this connection, see also Remark \ref{pvar}(c). 
  
(c) However, we do get an infinitary partial semigroup if 
in a Hausdorff  $T_3$ topological space  we define
an ``uppercase''   Limit by 
$\Lim  _{i \in I} a_i = a$
if both 
$\lim  _{i \in I} a_i = a$ (in the sense of (b) above)
and moreover 
$\lim  _{i \in K} a_i $
exists for every convex nonempty subset of $I$.
\end{example}

\begin{example} \labbel{omegaPP} 
The notion of an $ \omega$-semigroup from \cite[4.1]{PP} can be recast
as a particular case of a partial infinitary semigroup 
in the sense of Definition \ref{pnc}.
According to \cite{PP}, an \emph{$ \omega$-semigroup}
is a set $S$ which is the disjoint union
of two sets $S_+$ and $S_ \omega $
and some finite and $ \omega$-indexed products are defined.

 In detail, a finite product is defined
in case all of its factors belong to $S_+$, in which case
the product should belong to $S_+$.
A finite product is defined also 
in case all of its factors belong to $S_+$,
except for the last one; in this case, the product belongs to
 $S_ \omega $.
An $ \omega$-indexed products of elements
of $S_+$ is defined and should belong to $S_ \omega $.
All the elements of  $S_ \omega $ can be obtained in this last way.
 No other product is defined. 

The correspondence between the above definition and 
\cite[4.1]{PP} is immediate. 
Properties (N) and (U) imply Properties (1)-(4) in
\cite[4.1]{PP}; and conversely, it is easy to see that, 
assuming (U), 
Properties (1)-(4) in
\cite[4.1]{PP} imply (N) in its full generality.

An $ \omega$-semigroup in the above sense can be 
turned into a complete semigroup.
The shortest way is obtained by adding a new element $ \Omega$
and setting to $ \Omega$ all undefined products.
Full details will be given in Example \ref{addinf}.  
\end{example}

\begin{example} \labbel{omega1}
Finite \emph{$ \omega_1$-semigroups}
have been thoroughly studied in \cite{Be,BC}.
In the present terminology an $ \omega_1$-semigroup
is a  ${<} \omega_1$-semigroups, namely a partial infinitary
semigroup such that $\prod _{i \in I} a_i $ is defined
for every well-ordered set $I$ of cardinality $< \omega_1$.
In the terminology from 
\cite{Be,BC}
 \emph{finite} means only that the base set ($S$, in our notation)
of the semigroup is a finite set; of course, the structure is
formally infinite, since infinitely many products are supposed to be defined.
However, it is proved in the quoted  works that a 
finite $ \omega_1$-semigroups can be completely described
by some finite piece of information.
 \end{example}

\begin{example} \labbel{diamond}
In \cite{Ri,RC}  the notion of a
 $\diamond$-semigroups has been introduced; in the present terminology it is
a partial infinitary semigroup such that every product
indexed by a countable scattered linear ordering is defined.
It is shown that a finite  $\diamond$-semigroup
has again a finite description (as in the previous example,
 ``finite''  means that the base set is finite).
 \end{example}

\begin{example} \labbel{ord}
\emph{Ordinals.}
 The class of all ordinals, with the usual transfinite sum,
is an ordinal semigroup.
This is seen by the first displayed equation in 
 clause (4) in Bachman \cite[p. 51]{bach}.
\arxiv{The reader is kindly referred to, e.~g.,  
\cite{bach,sier}
for further information about ordinals and the operations on them. }
To match formally our general Definition \ref{pncx},
the operation is to be expanded to well-ordered index sets
using Remark \ref{conseqa}.  
\arxiv{See Remark  \ref{conseqord}. 

(b) Let $\lambda$ be any infinite cardinal and 
choose some ordinal $ \gamma $ 
such that any sum of $ < \lambda $ 
ordinals $ < \gamma  $ is still $ < \lambda $.
For example, we can take $\lambda = \omega _1$
and $\gamma = \omega_1 \omega_1 $.  
In particular, we can take 
$\lambda$ an infinite regular cardinal
and $\gamma= \lambda $. 

Then the set of all ordinals $< \gamma  $
with the usual infinitary sum (with $<\lambda$ summands)
is a  ${<}\lambda$-semigroup.
Again, formally, one should expand the operation as
in Remark  \ref{conseqord}. 

(c) As a ``minimal'' particular case of the above example, the set of countable ordinals
is a $ {<}\omega_1$-semigroup.  
 }
 \end{example}   

\arxiv{
Note that there are other ``natural'' infinitary operations  
on ordinals which fail to satisfy infinitary forms of associativity
\cite{Al,w,transf,natprod,another}. 
}

\begin{example} \labbel{strings} 
\emph{Transfinite strings.}
 The class of ordinal-indexed  strings of elements from
some nonempty set $X$, with (infinitary) string concatenation, 
is an ordinal  semigroup, 
as already noticed at least as early as in Karp \cite{Karp}.

In more detail, 
a (transfinite) \emph{string} $s$ 
is a function  $s: \alpha _ s\to X $,
for some ordinal $\alpha_s  $ which depends on $s$.
Of course, this is the same as an ordinal-indexed sequence
of elements of $X$; the term string is  generally used
when considered together with the operation of concatenation.  
If $s$ and $t$ are two strings,
their 
\emph{concatenation} 
$s*t$ is a string defined on 
$\alpha_s + \alpha _t$ by
\begin{equation*}\labbel{str}
(s*t)( \beta )     =
\begin{cases}
s( \beta )   &   \text{if  $ \beta < \alpha _s  $},\\  
 t( \beta' )    &  \text{if  $  \beta = \alpha _s  + \beta ' $ and $\beta' < \alpha _t$ }   \\ 
\end{cases} 
\end{equation*} 

More generally, if $(s_ \gamma ) _{ \gamma \in \delta } $
is a sequence of strings, defined, respectively,
on $( \alpha _ \gamma ) _{ \gamma \in \delta } $ 
then the transfinite concatenation $\prod _{ \gamma < \delta } s_ \gamma  $
is a string on   
 the ordinal sum 
$\sum _{ \gamma \in \delta } \alpha _ \gamma  $.
The definition proceeds by ordinal induction on $\delta$.

If $\delta= 0$, then
 $\prod _{ \gamma < \delta } s_ \gamma = \prod _ \emptyset $
is the empty string.

If $\delta= \varepsilon +1$,
then
$\prod _{ \gamma < \delta } s_ \gamma =
 \left( \prod _{ \gamma < \varepsilon  } s_ \gamma \right) * s_ \varepsilon  $.
 
If $\delta$ is limit, then 
$\prod _{ \gamma < \delta } s_ \gamma $
is the only function which extends each
$\prod _{ \gamma < \varepsilon  } s_ \gamma$,
for $\varepsilon < \delta $.  

There is a vast literature on ordinal-indexed strings,
especially in the countable case.
We refer to Choffrut and Horv{\'a}th \cite{CH} and references there.

Note that Example \ref{ord}  can be identified with a particular case 
of  the present example when $ \vert   X \vert  =1$. 

\arxiv{ 
(b) Under the assumptions in Example  \ref{ord}(b), 
the set of all strings of length $< \gamma  $ with elements from $X$,
with concatenation of $<\lambda $ many strings,  is 
a ${<}\lambda$-semigroup.
Note that in the case $\lambda = \gamma = \omega $
we get the usual (classical) semigroup of strings of finite length. 

(c)
The set of strings of countable length 
(i. e., of length $< \omega_1$)
with elements from $X$
is a $< \omega _1$-semigroup.
}
\end{example}

\begin{example} \labbel{ordex}   
\emph{Partially ordered sets (posets).}
 If $I$ is a linearly ordered set
and $(A_i) _{i \in I} $ is a sequence of pairwise disjoint 
posets,  the \emph{ordered sum}, sometimes called \emph{ordinal sum} 
$\sum _{i \in I } A_{i}$ is the poset
defined on the set  $\bigcup_{i \in I } A_i $ with the order relation defined by the following
condition.
  \begin{enumerate}
   \item[(OS)]  
\emph{$x \leq y$ if and only if either}
  \begin{enumerate}[(i)]    
  \item  
\emph{$x, y \in A_i$, for some $i \in I$, and
$x \leq y$ holds in $A_i$, or}
\item
 \emph{$x \in A_i$, $y \in A _{i'} $,
for some $i ,i' \in I$, and $i < i'$ holds in $I$.}
 \end{enumerate}
 \end{enumerate}

See  Harzheim  \cite[4.1]{Harz}, Rosenstein 
\cite[1, \textsection 4]{Ro} or  Sierpi{\'n}ski \cite[XII, 4]{sier} for further  details. 
The above operation can be used to construct a complete semigroup $S$.
Formally, $S$ is not  on the class of  
posets, but  on the class of their \emph{types}, that is,
their equivalence
classes modulo order isomorphism. Of course, one has to check
that the construction on  types 
is independent from the representatives, we refer again
to \cite{Harz,Ro,sier} for details. 

To check that $S$ is indeed a complete semigroup 
in the sense of Definition \ref{pncx} one has  
to verify that if $I$, $A_i$ are as above
and $\pi:I \to J$ is surjective and order preserving,
then $\sum _{i \in I } A_{i}$
is isomorphic to 
 $\sum _{j \in J } \sum _{\pi(i) =j } A_{i}$.

Note that Example \ref{ord} can be seen as the particular instance  of
the present example restricted to the case when both 
$I$ 
and the $A_i $s are well-ordered.
This is because the ordinals form a class of representatives
for types of well-orders, and ordinal sum allows a definition
in terms of ordered sums of well-orders.
See, e.~g., \cite[XIV, 3]{sier} 

\arxiv{
(b)
If $\lambda$ and $\kappa$ are infinite cardinals such that 
$cf( \kappa ) \geq \lambda $,
then the collection of  types of posets
of cardinality $<\kappa$ is 
a ${<}\lambda$-complete semigroup,
with the above operation of ordered sum.
Indeed, the ordered sum of $<\lambda$ many 
posets, each of cardinality $< \kappa$,
has  cardinality $< \kappa$.

(c)
In particular, 
 the set of  types of countable posets is
a countably complete (i. e., ${<} \omega _1$-complete) semigroup,
with the  operation of ordered sum of types. }
\end{example}

\begin{example} \labbel{labeledord} 
\emph{Labeled posets.}
We can modify the above example to deal with 
types of labeled partial
 orders rather than  orders.
A \emph{labeled poset} with labels from the set $X$ is a poset 
$A$  together with a function $f: A \to X$.
Of course, this is the same as an
$A$-indexed sequence of elements from $X$.

Types of labeled orders are defined in the usual way.
If $f: A \to X$, $g: B\to X$,
the corresponding labeled orders have the same \emph{type}
if there is an order-isomorphism
$h: A \to B$ such that $f(a)= g(h(a))$,
for every $a \in A$.
Note that if $A$ and $B$ are ordinals,
then $f$ and $g$ define the same type
if and only if $f= g$;
this is due to the fact that any ordinal has just one automorphism, the identity
function.
However, the situation might be different for orders 
with non trivial automorphisms. 
Say, if $A$ is the ordered set of the integers,
then $(x_z) _{ z \in A} $ and 
 $(x _{z+1} ) _{ z \in A} $
 represent the same type.
This identification is necessary, if we want Property (N) 
(or just Property (Iso) from Remark \ref{conseqa})
to be
satisfied.
 
Definition (OS) from \ref{ordex} can then be extended
 in the natural way, and it is standard
to see that it is representative-invariant.
 
In a sense,  Examples \ref{ord}, \ref{strings} and  \ref{ordex} above   can all be 
seen as particular cases of
the present example. This is obvious for \ref{strings}; 
indeed, an ordinal-indexed string is the same as a labeled ordinal.
We have already 
mentioned that \ref{ord} 
 can be seen as a particular case of \ref{strings}. As for \ref{ordex},
one can think of a poset as a labeled poset with  labels 
taken from a singleton.
 \end{example} 

\arxiv{
In fact,  Example \ref{labeledord}
(restricted to the case of linear orders) 
provides a universal construction of the ``free'' 
complete semigroup on some set $X$.
Recall the definition of a homomorphism from
Definition \ref{hom}.

\begin{proposition*} \labbel{free}
Suppose that $\lambda$ is an infinite regular cardinal,
$(S, \prod)$ 
is a  ${<}\lambda$-complete semigroup
and $X $ is a nonempty subset of $  S$. 

Let 
$(L, \sum)$ be  the ${<}\lambda$-complete semigroup of types of
 labeled linear orders of cardinality $<\lambda$ 
with labels from the set $X$, as in Example \ref{labeledord} above.

Then there is a unique homomorphism
$f: L \to S$ 
such that 
$f(\bar x)= x$,
where $\bar x$ 
is the type of the one-element order
with its only element labeled as $x$. 
 \end{proposition*}

  \begin{proof} 
Uniqueness is elementary.
Indeed, suppose that there exists such a homomorphism,
call it $f$. Let $\ell \in L$  be the type of a labeled 
ordered set on the order $I$
and labeled as $( x_i) _{i \in I} $. 
By construction,
$\ell= \sum _{i \in I} \bar x_i $ and,
since $f$ is supposed to be a homomorphism, 
\begin{equation}\labbel{blurp} 
    f(\ell)= f(\sum _{i \in I} \bar x_i) = \prod _{i \in I} f(\bar x_i)
=  \prod _{i \in I} x_i 
 \end{equation}
 by the conditions $f$
is supposed to satisfy.
Hence the conditions uniquely determine $f$. 
  
To prove existence,
we take the identity between
the second and the last term in equation \eqref{blurp} as a definition.
We have to verify
that if $(\ell_j) _{j \in J} $ are types of labeled orders,
then  
$f(\sum _{j \in J} \ell_j) = \prod _{j \in J} f(\ell _j)$.
Say, each  $\ell_j$ has  
the same order-type of $I_j$
and is labeled as $( x_i) _{i \in I_j} $,
where, as usual, the $I_j$s are pairwise disjoint. 
Then, setting $I= \sum _{j \in J} I_j $, 
we get the following chain of identities.
\begin{multline*}\labbel{tribba}
f(\sum _{j \in J} \ell_j) 
=^{\text{def}}
f(\sum _{j \in J} \sum _{i \in I_j} \bar x_i )
= ^{\text{(N)}}
f(\sum _{i \in I} \bar x_i )
= ^{ \eqref{blurp}}
\prod_{i \in I} x_i 
\\
= ^{\text{(N)}}
\prod_{j \in J} \prod _{i \in I_j}  x_i 
= ^{ \eqref{blurp}}
\prod_{j \in J} f(\sum _{i \in I_j}  \bar x_i) 
=^{\text{def}}
 \prod _{j \in J} f(\ell _j)
 \end{multline*}
where we denote by $=^{\text{def}} $ an identity which follows from
the above assumptions and  by $= ^{ \eqref{blurp}}$, $= ^{\text{(N)}}$
 an identity which follows from the corresponding equation. 
\end{proof} 
}

\begin{example} \labbel{rel}
More generally, we could consider the \emph{ordinal sum} 
of structures with a binary relation. 
See \cite[p. 60 and Theorem 2.5(ii)]{Tord}. 
A survey about this and related notions can be found in 
J{\'o}nsson \cite{Jo}, with many historical remarks.

Also the transfinite iteration of the ordinal sum
of combinatorial games constitutes an ordinal semigroup.
In detail, recall the definition of a \emph{combinatorial (``long'') game}
from, e.~g., \cite[Chapter VIII]{Sieg} and the definition of the
(binary) \emph{ordinal sum} $G : H$ of two games from
\cite[p. 89]{Sieg}.  If $(G_ \gamma  ) _{ \gamma < \alpha } $
is an ordinal-indexed sequence of combinatorial games, let us define
the \emph{transfinite ordinal sum} $\prod _{\gamma < \alpha } G_ \gamma $  
by induction as follows.
If $\alpha=0$,  $\prod _{\gamma < \alpha } G_ \gamma $ is the $0$
game in which no player has any move.
If $\alpha= \beta +1$ is successor,
$\prod _{\gamma < \alpha } G_ \gamma $ is the binary ordinal sum
$ \left( \prod _{\gamma < \beta  } G_ \gamma \right)  : G_ \beta $.
If $\alpha$ is limit, a player moves in 
$\prod _{\gamma < \alpha } G_ \gamma $
by choosing some successor ordinal $\beta< \alpha $
and making a move in the already defined game
$\prod _{\gamma < \beta  } G_ \gamma $. 

Note that the representation of a surreal number as a sign sequence
is an ordinal sum as above of $1$ and $-1$s.
 \end{example}

\begin{example} \labbel{ordprod}
\emph{Ordinals with product.}
 As another example, we can take the class of ordinals with infinitary (ordinal-indexed) products, thus getting an ordinal  semigroup with respect to $\prod$. See clause (4) in \cite[p. 51]{bach}. 
Then $(Ord, \sum, \prod)$ can be called an \emph{ordinal near-semiring},
since  both operations have a neutral element and  infinitary right-distributivity holds: 
$ \alpha  \sum _{ \delta \in \gamma } \beta _ \delta =
 \sum _{ \delta \in \gamma } \alpha  \beta _ \delta  $.
See clause (5) in \cite[p. 51]{bach}. 
 \end{example}

\section{Constructions} \labbel{const}

This section is still devoted to examples
of infinitary semigroups. Here we start with some given structure,
in most cases an arbitrary semigroup,   and 
construct a new semigroup  satisfying certain properties.
\arxiv{Sometimes the starting structure is not actually a semigroup, 
as in  Example \ref{incr}.}

First, we show that, under a weak  natural request,
there are many semigroups (in the classical sense)
that cannot even be extended  to an $ {\leq}\omega$-semigroup,
namely an infinitary semigroup such that all finite products
and all products of sequences with order-type $ \omega$ are defined. 
The next proposition  is modeled
after classical examples in the commutative case.
It shows that if some semigroup has an idempotent element
$e$ and two elements $a$, $b$ such that $ab=e=ba$,
 $ae=e$ and  $be=b$ ,
then we cannot even have $\prod _ \omega e = e$
in an extension which is an $ {\leq}\omega$-semigroup. 

Let us write $aaa\dotso$ in place of
$\prod _ \omega a$, and let similar abbreviations be in charge. 

\begin{proposition} \labbel{abeba} 
Suppose that $S$ is a partial infinitary semigroup, $a,b,e \in S$,
the product $abababa\dotso $ is defined and the  identities
$ab=e=ba$, $ae=a$, $be=b$  and  $eee\dotso=e$   hold.
Then $a=e=b$.

In particular, if some ${\leq} \omega $-semigroup
has some neutral element $e$ such that  
$\prod _ \omega e = e$, then no element
distinct from $e$ has a bilateral  inverse.
\end{proposition}  

\begin{proof} 
By using (N) quite heavily, we get both
$ababab\dotso = (ab)(ab)(ab)\dotso = eee\dotso=e$
and 
$abababa\dotso = a(ba)(ba)(ba)\dotso = aeee\dotso= a(eee\dotso)=ae=a$,
hence $a=e$. 

Since $abababa\dotso $ is defined, 
we get from
 Remark \ref{conseqc}
that 
$bababa\dotso $ is defined.
We can now perform a symmetrical argument to get
$b=e$.  
\end{proof}

\arxiv{
\begin{remark*} \labbel{abebarmk}
The assumption $be=b$  is necessary in Proposition \ref{abeba}.
As a counterexample,
just let $S= \{ \, b, e  \,\}$ with $e$ fully absorbing and 
$b$ fully idempotent. We get a complete infinitary semigroup
and, if we take  $a=e$, all the remaining identities are satisfied.
 \end{remark*}     
}

A typical example, or perhaps \emph{the} typical
example of a semigroup in the classical sense
is a set $S$ of  functions
from some set $X$ to itself, with $S$ 
 closed with respect to  the operation of composition.
Proposition \ref{abeba}  shows that, in general, such a semigroups $ S$ 
cannot be extended to a $ {\leq}\omega$-semigroup,
under some reasonable requests.
For example, this is the case
when $S$ contains the identity function $e$ and a nonidentical 
bijective function together with its inverse,
and we require that an  infinite power of $e$ 
still gives $e$.

 In particular, if $ \vert   X \vert   \geq 2$
 and $S$ contains \emph{all}
the functions from $X$ to $X$, 
then $S$  cannot be extended to a $ {\leq}\omega$-semigroup,
under the above request about powers of the
identity.
On the other hand,
a natural expansion is possible
when $S$ contains only particular kinds of functions.

Recall that a \emph{chain}  in a partially ordered set
is a linearly ordered subset. A partially ordered set
is \emph{(upper) chain complete} if
 every chain has a least upper bound.
It is \emph{chain ${<}\lambda$-complete} if
 every chain of cardinality $<\lambda$ has a least upper bound.
 For our purposes, it is not important whether or not we include the empty 
chain in the above condition.
It is well-known that there are other  conditions equivalent
to chain completeness, but chain completeness is
what we will actually  need  here.
See \cite{Mar} for further information.

\begin{example} \labbel{incr}
Suppose that $X$ is a chain-complete partially ordered set
 and let 
$F$ be the set of all the functions
$f:X \to X$ such that $f(x) \geq x$,
for every $x \in X$.

We will define
 inductively $\prod _{i < \delta } f_i $, for
every ordinal $\delta $ and every  sequence 
$( f_i) _{i \in \delta  } $ of members of $F$,
simultaneously checking that 
 $\prod _{i < \gamma   } f_i  \leq \prod _{i < \delta } f_i $,
for 
$ \gamma  \leq \delta $, in particular, also
 $ \prod _{i < \delta } f_i \in F$. 

If $\delta = 0 $, let  $\prod _ \emptyset  f_i = e $,
the identity function on $X$, noticing that $e \in F$. 

 If $\delta= \varepsilon +1$,
let
$\prod _{ i < \delta } f_ i=
 \left( \prod _{ i < \varepsilon  } f_ i \right) \circ  f_\varepsilon $,
where $\circ$ denotes \emph{composition}.
Since 
$f_ \varepsilon  \in F$,
we have 
 $\prod _{ i < \varepsilon  } f_ i \leq 
  \prod _{ i < \delta  } f_ i $.
By the inductive hypothesis, 
 $\prod _{ i < \gamma   } f_ i \leq 
  \prod _{ i < \varepsilon } f_ i $, 
for every $\gamma \leq \varepsilon $.
Hence 
 $\prod _{ i < \gamma   } f_ i \leq 
  \prod _{ i < \delta  } f_ i $, 
for every $\gamma \leq \delta  $.
 
If $\delta$ is limit, then let 
$\prod _{ i < \delta } f_ i = 
\sup _{ \varepsilon < \delta } \prod _{ i < \varepsilon  } f_i$.
This is well-defined,
since $X$ is chain complete and, by the inductive hypothesis,
$ \left( \prod _{ i < \varepsilon  } f_i \right) _{ \varepsilon < \delta } $ 
is a chain.
It is then straightforward   that
 $\prod _{ i < \gamma   } f_ i \leq 
  \prod _{ i < \delta  } f_ i $, 
for every $\gamma \leq \delta  $.

 It is not difficult to prove that 
$(F, \prod)$ is an ordinal semigroup.

If $\lambda$ is 
an infinite cardinal and 
$X$ is only assumed to be chain  ${<} \lambda$-complete,
the above induction can be carried over for every 
ordinal $\delta < \lambda $,
and in this way  $(F, \prod)$ is a  ${<} \lambda$-semigroup.
 \end{example} 

Note that, in the above example, if $f,g \in F$
and $f$ and $g$ are one the inverse of the other,
then necessarily $f=g$ is the identity function.
Thus Example \ref{incr} does not contradict 
Proposition \ref{abeba}.   

\begin{remark} \labbel{narr}
If we give up the assumption
that, say, infinite powers of some idempotent $e$  still coincide with
$e$, then every  semigroup in the classical sense
can be actually extended to a complete semigroup.
Just add a new absorbing element $ \Omega $ ``at infinity''
and set to $ \Omega$ the result of every infinite product.

The argument
has a general formulation 
which applies to
infinitary semigroups, as well.
We can turn any partial infinitary semigroup $S$ 
to a complete one by adding a new  element, provided
that $S$ satisfies the following property.

  \begin{enumerate}
   \item[\narr] \emph{Whenever  $\pi:I \to J$ is a surjective  order
preserving map
and  all the products on the right-hand side of the equation 
$\prod _{i \in I} a_i = \prod _{j\in J} \prod _{\pi(i) =j }  a_i$ 
 are defined,
then also the product on the left-hand side is defined 
(hence equality holds, by (N))}.
  \end{enumerate}

Note that every (finitary) semigroup in the classical sense satisfies \narr.
 \end{remark}

\begin{example} \labbel{addinf}
Suppose that $(S, \prod)$ is a partial infinitary semigroup satisfying
\narr\ and $ \Omega$ is a new element not in $S$.
Define $\prod'$ on $S'=S \cup \{ \Omega \} $  by the following rule.
\begin{equation}\labbel{cuc}   
\sideset{}{'}\prod _{i \in I} a_i  =\begin{cases}
\prod _{i \in I} a_i &    \text{if  $ \prod _{i \in I} a_i  $ is defined},\\
\Omega  &    \text{otherwise}
\end{cases}
\end{equation}     
 (in particular, the second clause applies, if $a_i= \Omega $,
for some $i \in I$).

Then $(S', \prod')$ is a complete semigroup extending 
$(S, \prod)$. 
\arxiv{Compare Definition \ref{hom}.

In particular,
every ordinal semigroup in the strict sense
(that is,  no other product is defined,
besides those products giving the ordinal structure)
can be extended
to a complete semigroup,
since it satisfies \narr.
The same applies to
${<}\lambda$-complete semigroups and 
${<}\lambda$-semigroups, for 
$\lambda$  an infinite regular cardinal.
This is generally true also for ${<}\gamma$-semigroups,
when $\gamma$ is an ordinal, as we will show in 
\cite{ordsmg}; 
however this is highly nontrivial, since,
when $\gamma$ is not a cardinal,
${<}\gamma$-semigroups in the strict sense
do not satisfy \narr\ (of course, in this case we will need
a definition more elaborate than \eqref{cuc}). }
 \end{example}

Of course, the assumption that \narr\ is satisfied is 
needed in order to perform the construction in 
Example \ref{addinf}. Indeed, if everything in
$ \prod _{j\in J} \prod _{\pi(i) =j }  a_i$  is defined
and we set $\prod _{i \in I} a_i = \Omega $,
then (N) necessarily fails, since $ \Omega \not \in S$.
However, \narr\ is not a necessary condition
for  extendability to 
a complete semigroup, 
\arxiv{as it follows already from the example
of $< \gamma$-semigroups mentioned in  \ref{addinf}.
See also the discussion in Remark \ref{pvar}(b).}

Let us say that an element 
$ \Omega$ in a partial infinitary semigroup 
is \emph{fully absorbing} if 
$\prod _{i \in I} a_i =  \Omega $,
whenever $\prod _{i \in I} a_i $  is defined and 
$a_i= \Omega $, for some
$i \in I$. 
Thus $ \Omega$ in 
Example \ref{addinf}
is fully absorbing.
As we will see in Theorem \ref{krob} below,
any set which is a complete commutative semigroup 
necessarily contains
a fully absorbing element. 
In the next examples 
we show that this is not necessarily the case for noncommutative
complete semigroups. This shows that the
noncommutative theory is 
somewhat richer and less trivial.

Of course, as soon as we allow some infinitary operation
together with some very weak forms of 
 associativity, we get that absorption phenomena necessarily occur.
E. g., if $w=aaa\dots$ is defined,
 then  necessarily $aw=a(aaa\dots)=aaaa\dots=w$,
under a poor man's version of (N).  

\begin{example} \labbel{left}
   Suppose that $S$ is a nonempty set and fix $s_0 \in S$.
For an ordered set $I$, define $\prod _{i \in I } a_{i}$ 
to be $s_0$, if $I$ has no minimum, and 
$a_i$, if $i$ is the minimum of $I$.
Then $S$ becomes a complete semigroup.  
 \end{example} 

In Example \ref{left},
$s_0 * a $ holds, for every $a \in S$,
but $a*s_0= s_0$ if and only if $a= s_0$,
thus $s_0$ is only partially absorbing.
Even lighter absorption phenomena occur in the next example.

\begin{example} \labbel{lr}
 Suppose that $S$ is a  semigroup (in the classical sense),
 $s_0, s_1 \in S$ and 
$S$ 
satisfies the identity 
$a*b*c=a*c$,
for every $a, b, c \in S$. 

Then $S$ can be made into a complete semigroup by letting, for 
every ordered set $I$:
\begin{equation*}
\prod _{i \in I } a_{i} =\begin{cases}
s_0 *  s_1 &    \text{if  $ I  $ has no minimum and no maximum},\\
a_i *  s_1 &    \text{if  $i$ is the minimum of  $ I  $ and $I$ has no maximum},\\
 s_0 * a_j &    \text{if  $ I  $ has no minimum and $j$ is the maximum of $I$}\\
a_i * a_j &    \text{if $i \not=j$,  $i$ is the minimum and $j$ is the maximum of $I$} \\
a_i  &    \text{if  $I= \{ i \} $ } \\
\end{cases}
   \end{equation*}

Note that the binary  operation on $S$ is preserved.
 \end{example} 

Note that  Example \ref{left}
is the particular case of  Example \ref{lr}  in the special 
case of a semigroup $S$ satisfying  $a*b=a$, for every $a, b \in S$.

We conclude this section by constructing further semigroups starting from
some given ones.

\begin{example} \labbel{subs}
Let $S$ be a semigroup (in the classical sense) and
let $\mathcal P(S)$ denote the set of all subsets of $S$.
One can turn $\mathcal P(S)$ into a complete semigroup by
defining the following
product, for every linearly ordered set $I$ and every  sequence 
$( X_i) _{i \in I} $ of subsets of $S$.

$\prod _{i \in I} X_i =  \{ s_1*s_2* \dots *s_n \mid
n >0, \  i_1 < i_2 \dots < i_n \in I \text{ and }
 s_1 \in X _{i_1}, \dots , \ s_n \in X _{i_n}   \}$ 

Note that, except for a
few very special cases,  the canonical inclusion
which sends $s \in  S$ 
to $\{ s \} \in \mathcal P(S) $ 
is \emph{not} a semigroup homomorphism, since, 
for example,
$s*t$ goes to 
$\{ s*t \}$, which is generally different
from 
$\{ s \} *\{ t \} = \{ s,t,s*t \}  $. 

Note that $S$, as a subset of itself,
is a fully absorbing element
in the above example.

In the above construction we can  consider only
 nonempty subsets of $S$, still getting a complete semigroup.
 \end{example} 

We can modify the above construction in such a way that
 $s \mapsto \{ s \} $ 
gives a homomorphism, if we  start with some
infinitary semigroup.

\begin{example} \labbel{subs2}
Let $(S, \prod)$  be a partial infinitary semigroup 
in the sense of Definition \ref{pnc}
and suppose that $(S, \prod)$ 
satisfies property \narr\ from Remark \ref{narr}. 
We can turn $\mathcal P(S)$ into a complete semigroup 
 by
defining the following
operation $\prod'$. For every linearly ordered set $I$ and every  sequence 
$( X_i) _{i \in I} $ of subsets of $S$,
 we set
$\prod' _{i \in I} X_i =
\{ \prod _{i \in I} s_i \mid s_i \in X_i \text{, for every } i \in I \text{ and  }  
\prod _{i \in I} s_i \text{ is defined} \} $. 

Thus, in case $\prod _{i \in I} s_i $ is never defined
for every choice of the sequence of the $  s_i \in X_i $, we  
get $\prod' _{i \in I} X_i = \emptyset $.
In particular, in contrast with the above example, here
 the empty set is a fully absorbing element.

The assumption that $(S, \prod)$ satisfies \narr\ is necessary,
since if $s_i \in X_i$,
for $i \in I$,
 $\prod _{i \in I} s_i $ is not defined,
but all the  products in the expression  $   \prod _{j\in J} \prod _{\pi(i) =j }  s_i$ 
are defined, then, according to the above definition,
we would have 
that
$   \prod _{j\in J} \prod _{\pi(i) =j }  s_i$ belongs to 
$   \prod _{j\in J} \prod _{\pi(i) =j }  X_i$,
but not necessarily to
 $\prod _{i \in I} X_i $.
 \end{example} 

\arxiv{
Standard algebraic constructions apply to 
partial infinitary semigroups.

\begin{example} \labbel{pointw}
(a) If $( S_z) _{z \in Z} $ is a sequence of partial infinitary semigroups,
then $\prod _{z \in Z} S_z $ becomes a partial infinitary semigroup,
with the product defined pointwise, in the sense that if
$ \bar{a_i} = (a_{i,z}) _{ z \in Z} $, for  $i \in I$, then 
$\prod _{i \in I} \bar a_{i} $ is defined if and only if     
each $\prod _{i \in I} a_{i,z} $ is defined, for $z \in Z$, with value, of course,
$ \left(  \prod _{i \in I} a_{i,z}  \right) _{ z \in Z} $.

If each $S_z$ is a complete (an ordinal) semigroup, then 
$\prod _{z \in Z} S_z $  is a complete (an ordinal) semigroup.
Ordinal and cardinal restricted properties are preserved for
products with appropriately small index set.

(b) If $(S, \prod)$ is a partial infinitary semigroup and
$ \emptyset \neq T \subseteq S$, then $T$ becomes a partial infinitary semigroup
with the product $\prod^T$, letting $\prod^T _{i \in I} a_i $ 
to be  defined (of course, equal to $\prod _{i \in I} a_i $)
 exactly when,  
for every convex subset $J$ of $I$, $\prod _{i \in J} a_i $ is 
defined and belongs to $  T$.
Here completeness is lost, unless $S$ is already complete
and the subset $T$ is closed under applications of $\prod$.

(c) Suppose that  $(S, \prod)$ is a partial infinitary semigroup,
$  T \neq \emptyset $ and $\varphi : S \to T$ is a function
such that

(*) for every linear order $I$ and sequences
$( a_i) _{i \in I} $, $( b_i) _{i \in I} $, if 
$\varphi ( a_i) = \varphi ( b_i)$, for every $i \in I$ and both 
$a=\prod _{i \in I} a_i $ and $b=\prod _{i \in I} b_i $ are defined,
then $\varphi(a)= \varphi (b)$.  

Then $T$ becomes a partial infinitary semigroup
by letting $\prod^T _{i \in I} c_i $ 
to be  defined if there is a sequence $( a_i) _{i \in I} $
such that  $\varphi ( a_i) =  c_i$, for every $i \in I$,
and $a=\prod _{i \in I} a_i $ is defined.
If this is the case, set $\prod^T _{i \in I} c_i  = \varphi (a)$.

Note that we are dealing with partial operations,
hence we need not necessarily assume that $\varphi$  is 
surjective. Of course, in case $\varphi$  is not surjective, we must also 
deal with condition (U), so that we must let $\prod^T _{i \in I} c_i $
be defined also in case $\vert I \vert=1$, in which case 
the value of $\prod^T _{i \in I} c_i $ is given by (U).
 \end{example} }

\section{Comparison with the commutative case} \labbel{commu} 

\begin{definition} \labbel{commut}
We say that a partial infinitary semigroup is (fully)
\emph{commutative} if, besides (N) and (U), it satisfies also the following  condition. 

(C) \emph{If $f:I \to J$ is a
bijection (not necessarily respecting the order) 
and $b _{f(i)}  = a_i$, for every $i \in I$, 
then 
$\prod _{i \in I} a_i $ is defined 
if and only if 
$\prod _{j \in J} b_j $ is defined and, if defined,
they are equal.}
\end{definition}   

If condition (C) is satisfied, we can disregard the orders 
on $I$, $J$ etc. Hence
a partial infinitary commutative semigroup  in the present terminology
is essentially the same as 
a $\Sigma$-algebra satisfying (P) and (U)
in the terminology of Hebisch  and    Weinert \cite[Section IV.1]{HW}.
Note that any one-element set carries exactly one order on it,  
hence property 
(U) here is essentially the same as property (U) from \cite{HW}.

In particular, any commutative complete monoid, a \emph{mono\"\i de complet}
in the terminology from  Krob \cite[Definition II.1]{K}, can be thought
of as a complete infinitary semigroup in the sense of 
Definition \ref{pncx}.
In particular, this is the case for 
every complete lattice (with either meet or join as an operation).
 \arxiv{See, e.~g., \cite{GLT}  for 
further information about complete lattices.}

There are many possible nonequivalent definitions of
an identity element already in the commutative case. See \cite{HW}.
We give here a definition which appears to be the 
strongest possible one. The next definition is intended to apply
to the noncommutative case, as well.  

\begin{definition} \labbel{complid}
If $S$ is a partial infinitary semigroup in the sense of 
Definition \ref{pnc},
we say that $e \in S$ is a \emph{complete identity}  
if the following condition holds.
  \begin{enumerate}
   \item[(Id)]
\emph{If $\prod _{i \in I} a_i $ is defined
and  the (not necessarily convex) subset $H = \{ i \in I \mid a_i \not= e \} $
of $I$ has the induced order, 
then $\prod _{i \in H} a_i $ is defined, too, and
$\prod _{i \in I} a_i = \prod _{i \in H} a_i $.}
  \end{enumerate}

Since (U) implies that if $ \vert  I \vert  =1 $ then
 $\prod _{I } e $ is defined and equal to $  e$,
we get from (Id) that if $e$ is a complete identity for $S$, then
 $\prod _{\emptyset } $ is defined and $\prod _{\emptyset } = e $.
In particular, if some complete identity exists,  it is unique.

Moreover, if $S$ has some complete identity, then, 
by the above arguments, a more general version of Property (N)
holds, to the effect that there is no need to assume that $\pi$ is surjective.
 \end{definition}   

\begin{example} \labbel{agge}
 Suppose that $(S, \prod)$ is a partial 
infinitary  semigroup,  choose some new
element $e$ 
which does  not belong to $S$ and let 
$S'= S \cup \{ e \} $.
Then a partial infinitary product $\prod'$
can be defined  on
$S'$ 
in such a way that  $e$ is a complete identity of  $(S', \prod')$ 
and, moreover, if $I  \not= \emptyset   $
and $( a_i) _{i \in I} $ is a sequence of elements in $S$,
then
$\prod _{i \in I} a_i $ is defined if and only if
$\prod' _{i \in I} a_i $ is defined and, in case they are defined,
they are equal.

To prove the above assertion, if
 $( a_i) _{i \in I} $ is a sequence of elements of $S'$, 
 let
$\prod' _{i \in I} a_i $ be defined
if and only if either

(i) $a_i=e$, for every $i \in I$ (including the case $I= \emptyset $),
in which case set $\prod' _{i \in I} a_i = e$, or

(ii) $\prod _{i \in H} a_i $ is defined in $(S, \prod)$,
where  $H = \{ i \in I \mid a_i \not= e \} $ is nonempty.
In this case set
$\prod' _{i \in I} a_i = \prod _{i \in H} a_i$. 

Note that if $\prod_ \emptyset  $ is not defined,
then $(S', \prod')$ is an extension of $(S, \prod)$, but if  
$\prod_ \emptyset  $  is defined, then
$\prod'_ \emptyset  \not= \prod_ \emptyset  $.
Moreover, if the starting semigroup $(S, \prod)$ already had a complete 
identity, say, $ e_{\Pi}$, then 
$ e_{\Pi}$ ceases to be an identity in $(S', \prod')$,
since, for example, $ e_{\Pi}= e * e_{\Pi}$. 
 \end{example}

\begin{example} \labbel{agge2}     
By combining Examples \ref{addinf} 
and \ref{agge}, we get that
every partial infinitary semigroup satisfying
\narr\ can be turned into a complete semigroup
with a complete identity by adding two new elements. 
The procedure can be done in any order.
 \end{example} 

Krob \cite[Proposition II.3]{K} 
shows that every complete commutative
semigroup $S$ has an absorbing element. Note
that the assumption that $S$ is a monoid is not used in Krob's argument.
In fact,
the argument gives the following more refined result.

\begin{theorem} \labbel{krob}  
\cite{K} If $ \omega < \kappa $,  $S$ is
a ${<} \kappa  $-complete commutative
semigroup and $ \vert  S \vert   <\kappa $,  then
there is  $ \Omega \in S$ which is a \emph{${<}\kappa$-absorbing element}
in the strong sense that
$\prod _{i \in I} a_i =\Omega$,
whenever $ \vert  I \vert   < \kappa $ and  
$ a_i=\Omega$, for some $i \in I$.  
 \end{theorem} 

\begin{proof} 
Let $J= S \times \mathbb N $,
take $b_j = s$, if $j=(s, n)$,
and put 
$\Omega = \prod _{j \in J} b_j $.
Note that 
$ \Omega$ is defined, since
$ \vert  J \vert   = \max( \vert  S \vert  , \omega ) < \kappa $. 

Then (N), 
(U) and (C)
trivially imply that 
$ \Omega = \prod _{j \in J} b_j= 
 \prod _{j \in J \setminus (a,0)} b_j * a=   \Omega * a$,
for every $a \in S$,  
since, by  definition, the element $a$ occurs infinitely many times
in the sequence $( b_j) _{j \in J} $,
hence 
$ \prod _{j \in J \setminus (a,0)} b_j =
 \prod _{j \in J} b_j =  \Omega $. 

Next, consider a product $\prod _{i \in I} a_i$
such that, say, $a_{\bar\imath } = \Omega$.
If $ I = \{  \bar\imath\}$,
then $\prod _{i \in I} a_i = a_{\bar\imath } = \Omega$
by (U).  
Otherwise, we have
$\prod _{i \in I} a_i = a_{\bar\imath }  * \prod _{i \in I, i \not=\bar \imath} a_i 
= \Omega  * \prod _{i \in I, i \not=\bar \imath} a_i  = \Omega $,
by the previous paragraph, taking
$a=  \prod _{i \in I, i \not=\bar \imath} a_i$.  
\end{proof}

Note that the assumption
$ \omega < \kappa $ in Theorem \ref{krob}
is necessary, since every finite group
is (can be thought of as) a ${<} \omega $-complete semigroup,
but only one-element groups have an absorbing element. 

Statements  analogous to Theorem \ref{krob} 
fail in the  case of (noncommutative) ordinal semigroups,
as can be witnessed, e.~g., in Examples \ref{left} and \ref{lr}.
It is an open problem whether some weaker version of   
Theorem \ref{krob} holds also in the
noncommutative case. 

\arxiv{
Note that the assumption
$ \vert  S \vert   <\kappa $
is optimal in Theorem \ref{krob}.
The  naturals,
or even the integers form a commutative
${<} \omega $-semigroup with no absorbing element.
If $\kappa$ is a weakly (strongly) inaccessible cardinal,
then the set of all cardinals $\leq \kappa $,
with cardinal sum (cardinal product), is a commutative
${<} \kappa  $-semigroup with no absorbing element.
 If we are willing to accept proper classes
as the base set for an infinitary semigroup,
the class of all cardinals is a complete semigroup
with no absorbing element, either with
cardinal sum or cardinal product.

Example \ref{agge}  shows that we can always add
some ``external''  complete identity to an infinitary semigroup,
allowing the value of  $\prod _{\emptyset }$ to be changed. }
Proposition \ref{abeba}
shows that there are many obstructions which prevent  adding
some infinitary structure to, say, a finitary monoid,
in case we want to expand the operation only, without adding further elements.
In particular, by Proposition \ref{abeba},  
in the commutative case we cannot have an infinitary identity 
$e$ together with two elements $a$, $b$   such that $ab=e$, apart from the trivial case
$a=b=e$. 
The next theorem  shows that 
this can happen in the noncommutative case, even allowing  $e$
to be  a complete identity. 

\begin{theorem} \labbel{abe}  
There is  a  complete semigroup $(S, \prod)$ 
such that $S$ is a countable set and there are
elements $a,b, e \in S$  such that 
$ a  \neq e$, $ b \neq e$,
$ab=e$ and
 $e$ is a complete identity in the strong sense of
Definition \ref{complid}. 
\end{theorem}

See \cite{aberef}  for a proof.
Theorem \ref{abe} shows that in a complete infinitary
semigroup some elements might have either a left or
a right inverse. Of course, no nonneutral element 
has both  a left and
a right inverse, since just ternary associativity implies 
that such elements would be equal, but this is forbidden
by Proposition \ref{abeba}. 
Moreover, we cannot go too far, in particular,
it is not the case that \emph{every}  element has either a left
or right inverse.

\begin{proposition} \labbel{notut}
If $S$ is a partial infinitary semigroup,
$a,c,e \in S$, $a \neq e$, $e$ is a neutral element
(just in the finitary sense) and  
 the product $aaaa\dots\ c $ of length $ \omega+1$  is defined, then
it is not the case that $aaaa\dots\ c =e$.

In particular, if all ternary products are defined in $S$, 
$e$ is neutral and some product 
$aaaa\dots $ is defined, then
$aaaa\dots $ has not a right inverse, unless $a=e$. 
 \end{proposition} 

 \begin{proof}
If $aaaa\dots\ c =e$, then, by applying (N), we  get
$e=aaaa\dots \ c = ^{\textit (N)} a(aaa\dots\ c)=ae=a $. 
Strictly formally, the second statement in the proposition 
does not
follow from the first, since it might happen that  
$aaaa\dots\ c $ is not defined in $S$.
However, were $c$ a right inverse
of  $aaaa\dots $, it is enough to apply (N) also to  a $3$-product,
getting 
$e=(aaaa\dots ) c = ^{\textit (N)} (a(aaa\dots)) c = ^{\textit (N)}
a((aaa\dots) c) = ae=a $.
\end{proof}

\section{Problems} \labbel{prob} 

Since we have provided a lot of examples
of partial infinitary semigroups,
the notion is surely of interest.
We briefly discuss here possible variations on the notion.

\begin{remark} \labbel{pvar}
(a)
On one hand, Property (N)
in Definition \ref{pnc}
is relatively weak
and one could suggest to add some further properties to it;
for sure, Property \narr\ from Remark \ref{narr} looks like a natural request. 
However, most of the results we have proved in the present note
follow just from (N) and (U).

(b) A problem which, so far, is not completely solved is
under which conditions an infinitary semigroup can be
extended to a complete infinitary semigroup.
A sufficient condition \narr\  has been provided in 
Example \ref{addinf} but, as already pointed out
there, we will show in \cite{ordsmg} that the condition
is not necessary. 

Clearly, if some   infinitary semigroup $S$ can be
extended to a complete infinitary semigroup, then $S$
satisfies the following condition.

  \begin{enumerate}
   \item[(Eq)]
\emph{The following identity holds
\begin{equation*}\labbel{Qeq}
 \prod _{j\in J} \prod _{\pi(i) =j }  a_i
 = \prod _{h \in H} \prod _{\pi'(i) =h }  a_i
  \end{equation*}     
whenever $\pi: I \to J$ and $\pi': I \to H$
are surjective and all the  products in the equation are defined.}
  \end{enumerate}

We expect that Condition (Eq) is not sufficient for 
extendability  to a complete infinitary semigroup and instead we need 
a ``multiproduct'' generalization of  (Eq).

(c) In another direction, it seems that Property (N)
can be split into two different properties, and that 
there are natural classes of structures
satisfying exactly one of these properties.
In detail, let (Nmax) be the condition
asserting that all the instances of (N) hold
in the case when $J$ has a maximum. Symmetrically, 
let (Nlim) be the instance of (N)
which deals with the cases when $J$ has no maximum.

One can also consider the symmetric conditions dealing with minima, instead.

(d) There are  situations in which (Nlim)
is satisfied, but (Nmax) is not satisfied,
for example, for infinitary natural sums of ordinals,
see Proposition 2.4(6) in \cite{w} and the comment shortly
after its proof.
The converse pattern occurs in a very general situation,
as we are going to see.
 \end{remark}   

\begin{definition} \labbel{inflim}    
Let $Q$ be a complete meet semilattice such that 
every  chain has a supremum. For example, 
the class $Q$ of ordinal-indexed strings with elements 
 from some set $X$, letting $s \subseteq  t$ if
$s$ is an initial substring of $t$, is a    
complete meet semilattice such that 
every  chain has a supremum.
Under the above assumptions on $Q$,
for every $I$-indexed sequence $( q_i) _{i \in I} $
of elements of $Q$, define 
the \emph{inferior limit}
$\inflim _{i \in I} q_i $ as 
  $\bigvee _{i \in I} \bigwedge _{h \geq i} q_h $.
This is a special case of a well-known definition 
\cite[Definition II-1.1]{CLD}.

Under the assumptions, $\inflim$ is defined,   
since if $i \leq i'$, then   
$\bigwedge _{h \geq i} q_h \leq  \bigwedge _{h \geq i'} q_h$,
so that, letting $i$ vary, the set of the   
$\bigwedge _{h \geq i} q_h$ is a chain.
If $I = \emptyset $ we set  $\inflim _ \emptyset $
to be the infimum of $Q$.
The terminology is justified since if 
$Q = \mathbb R \cup \{ -\infty, \infty \} $,
then we actually get a (generalization of) the inferior limit. 

If $I$ has a maximum $\bar{\imath}$, then 
$\inflim _{i \in I} q_i  = q _{\bar{\imath}} $.
In particular, the ``finitary version'' of $\inflim$
corresponds to taking the last element of a sequence;
essentially, it corresponds to the semigroup operation given by
$a * b=b$. 
Otherwise, $I$ has not 
a maximum and $\inflim _{i \in I} q_i $ is the largest element
of $Q$ which is ``eventually compatible'' with the $q_i$.  

For example, in the case of  ordinal-indexed strings
(or, more generally, families of strings indexed by a 
linearly ordered set $I$),
the \emph{string limit} $\inflim _{i \in I} s_i $ is the common prolongment
of all the strings $t$ such that $t$  is eventually a substring of the 
$s_i$. Say, the string limit of 
the sequence $abcc, ababcc, abababcc, \allowbreak ababababcc, \dots$
is the infinite string $ababababab\dots$, where juxtaposition denotes string
concatenation. The string limit is always defined, but
the limit might be a very short string; for example, if the first
element of the strings $s_i $ is not eventually constant, then 
the limit is the empty string. 
  This  string limit has been studied in \cite{LM} 
with correlations to surreal numbers. 
\arxiv{See also 
\cite[Section 5]{series}. }
\end{definition}

The just introduced   $\inflim$, considered
as an infinitary operation, 
 does not satisfy (N).
For example, as in a remark above, 
the string limit of the alternating sequence
$a, b, a, b, a, \dots $ of length $ \omega$ is the empty  
string. If we let $\pi: \omega \to  \omega$
send $i$ to  the integer part of $i/2$, then  
 $\inflim _{ \pi(i) =j} a_i  = b $, for every $ j \in \omega$,
thus  $\inflim _{ j \in \omega } \left( \inflim _{ \pi(i) =j} a_i \right) =b 
\neq \emptyset = \inflim _{ i \in \omega } a_i $. 

Note that the above argument also shows that
$\inflim$ cannot arise as a topological limit;
compare Example \ref{limits}(a). Indeed,  at least for
$\leq \omega $-indexed sequences, the latter limit does
satisfy (N). While the fact that  $\inflim$
does not satisfy (N) might be seen as a drawback,
  $\inflim$ has the advantage of being everywhere defined and, in any case,
it satisfies at least (Nmax) from Remark \ref{pvar}(c).

Turning to another argument, in \cite{trieste}
we associated some invariants to a partial commutative infinitary semigroup
with a specified subclass,
obtaining some (rather easy) topological and
set-theoretical consequences.

Can we generalize the notions from
\cite[Section 3]{trieste} to noncommutative partial infinitary semigroups?

As another problem, can we consider a notion of an infinitary
semigroup providing ``products'' of ``net'' sequences indexed  
by just a partially ordered set? Note that we essentially provided examples:
Examples \ref{limits}, \ref{ordex}, \ref{labeledord}
and \ref{rel}    can be generalized in the present
context. 

\arxiv{
In order to keep the following  list within a reasonable length,
in many cases we have  cited survey works
in place of the original sources.
The reader is advised
to consult the quoted works 
for further references and, in particular, credits for original discoveries.
}


\section{Appendix I. Ordinal-indexed sequences are enough (for well-ordered
index sets)} \labbel{subse} 

 We now expand on Remark \ref{conseqa}. 
We will treat here explicitly 
the particular case dealing with ordinal semigroups,
since we have used it in order to construct some examples,
e.~g., \ref{ord} and \ref{incr}. Moreover, 
we will make heavy use 
of this shorter construction in \cite{ordsmg}.
It is appropriate to fix all the details 
because the arguments in  \cite{ordsmg} are quite delicate
and
some risk of circular reasoning is present anyway,
since any infinite well-ordered set 
has a proper subset isomorphic to itself.
Hence we make sure to fully fix every detail,
and we will employ
the present appendix for that purpose.
However, the following discussion  is essentially some sort 
of a triviality; if the reader 
is convinced by the above informal arguments, he might skip
the present section.

\begin{proposition} \labbel{conseqb}
Suppose that $S$ is a class and $\prod$ is a partial
class operation defined on ordinal-indexed sequences
of elements of $S$. Suppose that 
$\prod$ satisfies Property (U) 
from \ref{pnc}
 when $\vert I \vert =1$, as well as the following condition.
\begin{enumerate}
   \item[(Ord)]
\emph{
Assume that  $\prod _{ \alpha  < \delta } a_\alpha  $ is defined,
 $\pi: \delta  \to \eta$ is a surjective  order preserving map  and,
 for every $j < \eta$,
let $I_j = \{ \alpha \in \delta \mid \pi ( \alpha ) = j \}
= [\beta_j, \beta_j + \varepsilon_j )  $.
Whenever the above assumptions are met,
we require that  
all the products  in
the following equation are defined
\begin{equation*}
\prod _{ \alpha  < \delta } a_ \alpha  = \prod _{j <\eta} 
\prod _{ \gamma < \varepsilon_j } a _{\beta_j + \gamma }
  \end{equation*}     
and equality actually holds.
}
   \end{enumerate} 

Let us  define $\prod'$ on $S$ by the following 
condition. If $I$ is well-ordered,
 $\delta$ is the unique ordinal isomorphic to $I$
 and $g: \delta \to I$
is the canonical order isomorphism, then 
 $\prod' _{i \in I} a_i $ 
is defined if and only if 
 $\prod _{ \alpha  < \delta } a_{g( \alpha )} $ 
is defined and, if this is the case,
 $\prod' _{i \in I} a_i $ 
is given the same value.

Then $(S,\prod')$
is  a partial infinitary semigroup
in the sense of Definition \ref{pnc}.
Moreover $\prod$ and $\prod'$
coincide on the class of all ordinal-indexed sequences, 
in the sense that one is defined if and only if 
the other is defined and, when defined, they give the
same outcome.
\end{proposition}  

\begin{proof}
First notice that the notations in 
Condition (Ord)   make sense, since, for every $j \in \eta$, the set
$I_j = \{ \alpha \in \delta \mid \pi ( \alpha ) = j \}$
 is a convex subset of
$\delta$, hence it has the form 
$[\beta_j, \xi_j )  $
for some  $\beta_j < \xi_j$.
Here the inequality is strict, since 
$\pi$ is surjective, hence the counterimage of each point is nonempty.
 By ordinal arithmetic,
there is a unique nonzero ordinal 
$\varepsilon_j$ such that  
$\xi_j =  \beta_j + \varepsilon_j$, hence 
 $I_j $ has actually the form 
$[\beta_j, \beta_j + \varepsilon_j )  $, for some 
$\beta_j$ and $\varepsilon_j \neq 0$. 
Here $+$ denotes \emph{ordinal sum}.  

To prove the proposition, first notice that
$\prod'$ satisfies (U) trivially, since
$\prod$ is supposed to satisfy (U).
If $I$ is an ordinal, then $g$ is the identity, thus
$\prod'$ and $\prod$ coincide on ordinal-indexed sequences.   
Condition (N) remains to be proved.
If $\prod' _{i \in I} a_i $
is defined,  
for some well-ordered sequence, then,
by the definition
of $\prod'$,
also
$\prod _{ \alpha < \delta } a_{g( \alpha )} $
is defined and assumes the same value,
where $\delta$ and $g$ are as in the definition of $\prod'$. 
If  $\pi': I \to J$ is an order preserving map onto
a well-ordered set $J$ of order-type $\eta$,
then   $\pi'$ induces an order preserving map
$\pi: \delta \to \eta$.
By (Ord) and under the same notations there, 
$\prod _{ \alpha  < \delta } a_{g( \alpha )} = \prod _{j < \eta} 
\prod _{ \gamma < \varepsilon_j } a _{g(\beta_j + \gamma) }
$ and all factors are defined.
Now, for every $j\in \eta$, we have that  $\varepsilon_j$
is order-isomorphic to   
$I_j = \{ \alpha \in \delta \mid \pi ( \alpha ) = j \}
= [\beta_j, \beta_j + \varepsilon_j )  $.
If $j$ and  $j'$ correspond through the canonical 
isomorphism between   $\eta$ and  $J$, then the correspondence 
between $\pi$ and 
 $\pi'$ induces an order preserving map between
$I_j$ and  
$I_{j'} = \{ i \in I \mid \pi' ( i) = j' \}$, thus
$I_{j'} $ and 
 $\varepsilon_j$ are isomorphic.
But the latter is an ordinal, hence
the definition of $\prod'$
implies that  
$\prod' _{ \pi'(i) = j'} a_i $ 
is defined and equals 
$\prod _{ \gamma <\varepsilon_j } a _{g(\beta_j + \gamma) }$,
since the latter is defined by (Ord).
Thus  $\prod' _{ \pi'(i) = j'} a_i $ is defined for every $j' \in J$.
Then the above correspondence (which holds
for every $j' \in J$),  the definition of $\prod'$,
  and the fact that 
$ \prod _{j< \eta} 
\prod _{ \gamma < \varepsilon_j } a _{g(\beta_j + \gamma) }
$ is defined, again by (Ord), imply that
$\prod'_{j' \in J}\prod' _{ \pi'(i) = j'} a_i $ 
is defined. Moreover (Ord) and all the above correspondences
imply that 
 $\prod' _{i \in I} a_i = \prod_{j' \in J}\prod' _{ \pi'(i) = j'} a_i $.
Since $I$ and $\pi'$ were arbitrary we have proved that 
(N) holds for $\prod'$.  
\end{proof} 

Of course, we could have defined an ordinal semigroup by
restricting ourselves to ordinal-indexed sequences and 
 imposing only conditions (U)
and (Ord). However,
the definition from \ref{pncx}
has two advantages. First, the analogy with
complete (or partial) infinitary semigroups is clearer---the 
point being that, generally, for arbitrary types of linearly
ordered sets we have no preferred representative.
Moreover, the formulation of (N) 
is slightly easier in comparison with 
(Ord), in that $I_j$ is automatically well-ordered, though  not 
necessarily an ordinal.   
On the other hand, notice that, formally, the operation of 
a ${<}\gamma $-semigroup
 in the sense of Definition \ref{pncx}(II)
is always a proper class even
when $S$ is a set. Thus using only conditions
(U) and (Ord)
would be foundationally cleaner.

We have constructed ordinal semigroups  
proving (Ord), but then it is easier to use
(N) when working with products.
For example, we have been free to write
$\prod _{\pi(i) =j }  a_i$ in place
of  more involved expressions like
$\prod _{ \gamma < \varepsilon_j } a _{\beta_j + \gamma }$.
Note that, formally, if $\pi : \delta \to \eta$ is surjective
and order preserving, then 
$I_j = \{ \alpha \in \delta \mid \pi ( \alpha ) = j \} $
is not an ordinal, unless $j=0$.
 The shift from one notation to the other
 is possible by 
Proposition \ref{conseqb}.

\end{document}